\documentclass[reqno, 12pt, a4paper]{amsart}

\makeatletter
\let\origsection=\section \def\section{\@ifstar{\origsection*}{\mysection}}
\def\mysection{\@startsection{section}{1}\z@{.7\linespacing\@plus\linespacing}{.5\linespacing}{\normalfont\scshape\centering\S}}
\makeatother
\usepackage{amsmath,amssymb,amsthm}
\usepackage{mathabx}\changenotsign
\usepackage{bbm}
\usepackage{mathrsfs}
\usepackage{dsfont}
\usepackage[utf8]{inputenc}
\usepackage[babel]{microtype}

\usepackage{datetime}
\usepackage{fullpage}
\usepackage{setspace}
\usepackage{tikz}
\usetikzlibrary{decorations.pathreplacing,math}

\usepackage{xcolor}
\usepackage[backref]{hyperref}
\hypersetup{
    colorlinks,
    linkcolor={red!60!black},
    citecolor={green!60!black},
    urlcolor={blue!60!black},
}

\usepackage{bookmark}

\usepackage[abbrev,msc-links,backrefs]{amsrefs}
\usepackage{doi}

\renewcommand{\PrintDOI}[1]{\doi{#1}}

\usepackage[T1]{fontenc}
\usepackage{lmodern}

\usepackage[english]{babel}
\numberwithin{equation}{section}

\usepackage{enumitem}
\def\rmlabel{\upshape({\itshape \roman*\,})}

\let\polishlcross=\l
\def\l{\ifmmode\ell\else\polishlcross\fi}

\let\emptyset=\varnothing
\let\setminus=\smallsetminus

\makeatletter
\def\moverlay{\mathpalette\mov@rlay}
\def\mov@rlay#1#2{\leavevmode\vtop{   \baselineskip\z@skip \lineskiplimit-\maxdimen
   \ialign{\hfil$\m@th#1##$\hfil\cr#2\crcr}}}
\newcommand{\charfusion}[3][\mathord]{
    #1{\ifx#1\mathop\vphantom{#2}\fi
        \mathpalette\mov@rlay{#2\Cr#3}
      }
    \ifx#1\mathop\expandafter\displaylimits\fi}
\makeatother

\DeclareFontFamily{U}  {MnSymbolC}{}
\DeclareSymbolFont{MnSyC}         {U}  {MnSymbolC}{m}{n}
\DeclareFontShape{U}{MnSymbolC}{m}{n}{
    <-6>  MnSymbolC5
   <6-7>  MnSymbolC6
   <7-8>  MnSymbolC7
   <8-9>  MnSymbolC8
   <9-10> MnSymbolC9
  <10-12> MnSymbolC10
  <12->   MnSymbolC12}{}
\DeclareMathSymbol{\powerset}{\mathord}{MnSyC}{180}

\usepackage{tikz}
\usetikzlibrary{decorations.markings}
\usetikzlibrary{calc,positioning,decorations.pathmorphing,decorations.pathreplacing}

\makeatletter
\def\namedlabel#1#2{\begingroup
    #2%
    \def\@currentlabel{#2}%
    \phantomsection\label{#1}\endgroup
}
\makeatother

\newtheorem{theorem}             {Theorem}[section]
\newtheorem{lemma}      [theorem] {Lemma}

\newtheorem{observation} [theorem] {Observation}
\newtheorem{property}   [theorem] {Property}
\newtheorem{definition} [theorem] {Definition}
\newtheorem{proposition}[theorem] {Proposition}
\newtheorem{corollary}  [theorem] {Corollary}

\newtheorem*{claim}{Claim}

\let\theta=\vartheta
\let\rho=\varrho
\let\phi=\varphi

\usepackage{cases}

\newcommand{\mr}{\xrightarrow{\text{\rm mr}}}

\newcommand{\mF}{m_{\operatorname{F}}}
\newcommand{\rc}{\operatorname{r_c}}
\newcommand{\Q}{\operatorname{\mathcal{Q}}}
\newcommand{\cP}{\operatorname{\mathcal{P}}}
\newcommand{\PP}{\operatorname{\mathbb{P}}}

\usepackage[mathlines]{lineno}
\usepackage{etoolbox} 

\newcommand*\linenomathpatch[1]{%
   \expandafter\pretocmd\csname #1\endcsname {\linenomath}{}{}%
   \expandafter\pretocmd\csname #1*\endcsname{\linenomath}{}{}%
   \expandafter\apptocmd\csname end#1\endcsname {\endlinenomath}{}{}%
   \expandafter\apptocmd\csname end#1*\endcsname{\endlinenomath}{}{}%
 }
\newcommand*\linenomathpatchAMS[1]{%
    \expandafter\pretocmd\csname #1\endcsname {\linenomathAMS}{}{}%
    \expandafter\pretocmd\csname #1*\endcsname{\linenomathAMS}{}{}%
    \expandafter\apptocmd\csname end#1\endcsname {\endlinenomath}{}{}%
    \expandafter\apptocmd\csname end#1*\endcsname{\endlinenomath}{}{}%
}

\expandafter\ifx\linenomath\linenomathWithnumbers
\let\linenomathAMS\linenomathWithnumbers
\patchcmd\linenomathAMS{\advance\postdisplaypenalty\linenopenalty}{}{}{}
\else
\let\linenomathAMS\linenomathNonumbers
\fi

\linenomathpatchAMS{gather}
\linenomathpatchAMS{multline}
\linenomathpatchAMS{align}
\linenomathpatchAMS{alignat}
\linenomathpatchAMS{flalign}
\linenomathpatchAMS{numcases}
\linenomathpatch{equation}

\makeatletter
\newcommand{\definetitlefootnote}[1]{%
  \newcommand\addtitlefootnote{%
    \makebox[0pt][l]{$^{*}$}%
    \footnote{\protect\@titlefootnotetext}
  }%
  \newcommand\@titlefootnotetext{\spaceskip=\z@skip $^{*}$#1}%
}
\makeatother

\begin{document}
\onehalfspace
\shortdate
\yyyymmdddate
\settimeformat{ampmtime}

\definetitlefootnote{An extended abstract of this
  work~\cite{CoKoMoMo21+} has appeared in the proceedings of LAGOS
  2021.}

\title[Constrained Ramsey property]{The threshold for the constrained Ramsey property\addtitlefootnote}

\author[M.~Collares]{Maur\'icio Collares}
\address{Departamento de Matem\'atica, Universidade Federal de Minas Gerais, Belo Horizonte, MG, Brazil}
\email{mauricio@collares.org}

\author[Y.~Kohayakawa]{Yoshiharu Kohayakawa}
\address{Instituto de Matem\'atica e Estat\'{\i}stica, Universidade de
  S\~ao Paulo, Rua do Mat\~ao 1010, 05508-090 S\~ao Paulo, Brazil}
\email{yoshi@ime.usp.br}

\author[C. G.~Moreira]{Carlos Gustavo Moreira}
\address{School of Mathematical Sciences, Nankai University, Tianjin
  300071, P. R. China \& IMPA, Estrada Dona Castorina 110, Jardim Bot\^anico, Rio de Janeiro, RJ, Brazil}
\email{gugu@impa.br}

\author[G. O.~Mota]{Guilherme Oliveira Mota}
\address{Instituto de Matem\'atica e Estat\'{\i}stica, Universidade de
  S\~ao Paulo, Rua do Mat\~ao 1010, 05508-090 S\~ao Paulo, Brazil}
\email{mota@ime.usp.br}

\thanks{\rule[-.2\baselineskip]{0pt}{\baselineskip}%
  M.~Collares was partially supported by CNPq (406248/2021-4).
  Y.~Kohayakawa was partially supported by CNPq
  (311412/2018-1, 423833/2018-9, 406248/2021-4) and FAPESP (2018/04876-1,
  2019/13364-7).
  C. G. Moreira was partially supported by CNPq and FAPERJ.
  G.~O.~Mota was partially supported by CNPq (306620/2020-0,
  406248/2021-4) and FAPESP (2018/04876-1, 2019/13364-7).
  This study was financed in part by CAPES, Coordenação de
  Aperfeiçoamento de Pessoal de Nível Superior, Brazil, Finance Code~001.
  FAPESP is the S\~ao Paulo Research Foundation.  CNPq is the National
  Council for Scientific and Technological Development of
  Brazil.%
}

\begin{abstract}
  Given graphs $G$, $H_1$, and $H_2$, let $G\mr (H_1,H_2)$ denote the
  property that in every edge colouring of $G$ there is a
  monochromatic copy of $H_1$ or a rainbow copy of $H_2$. The
  \emph{constrained Ramsey number}, defined as the minimum $n$ such
  that $K_n\mr (H_1,H_2)$, exists if and only if $H_1$ is a star or
  $H_2$ is a forest.  We determine the threshold for the property
  $G(n,p)\mr (H_1,H_2)$ when $H_2$ is a forest, explicitly when the
  threshold is $\Omega(n^{-1})$ and implicitly otherwise.
\end{abstract}

\maketitle

\section{Introduction}
\label{sec:intro}
Given graphs $G$, $H_1$, and $H_2$, we write $G\mr (H_1,H_2)$ if in
every colouring of $E(G)$, with no restriction on the number of used
colours, there is a \emph{monochromatic} copy of $H_1$ or a
\emph{rainbow} copy of $H_2$, that is, a copy of~$H_1$ with all edges
having the same colour or a copy of $H_2$ with no two edges of the
same colour.  We investigate the property $G\mr (H_1,H_2)$ when $G$ is
the binomial random graph~$G(n,p)$.

We say a function $\hat{p} \colon \mathbb{N} \to [0,1]$ is a \emph{threshold
  function} for a property of graphs~$\mathcal{P}$ if
$\lim_{n \to \infty} \mathbb{P} [G(n,p) \in \mathcal{P}]=1$ for
$p \gg \hat{p}$, and
$\lim_{n \to \infty} \mathbb{P}[G(n,p) \in \mathcal{P}]=0$ for
$p \ll \hat{p}$.  We say that any~$p'=\Theta(\hat p)$ is `the
threshold' for~$\mathcal P$.  We say that a property~$\cP$ \emph{holds
  with high probability} for~$G(n,p)$
if~$\lim_{n\to\infty}\PP[G(n,p)\in\cP]=1$.  Thus, to show
that~$\hat p$ is the threshold for~$\cP$, we have to prove the
so-called \emph{$1$-statement}, i.e., that if~$p\gg\hat p$, then~$\cP$
holds with high probability, as well as the \emph{$0$-statement},
i.e., that if $p\ll\hat p$, then the complement of~$\cP$ holds with
high probability.

The \emph{constrained Ramsey number}~$\rc(H_1,H_2)$, sometimes called
the \emph{rainbow Ramsey number}, is defined as the minimum $n$ such
that $K_n\mr (H_1,H_2)$.  It is known (see, e.g.,~\cite{JJL}) that the
number $\rc(H_1,H_2)$ exists if and only if $H_1$ is a star or $H_2$
is a forest.  Therefore, assuming that $H_1$ is a star or $H_2$ is a
forest, we see that the property $G(n,p)\mr (H_1,H_2)$ admits a
threshold function because it is a non-trivial increasing
property~\cite{BT}.  Given graphs~$H_1$ and~$H_2$ for which
$\rc(H_1,H_2)$ exists, we denote by $\hat p(H_1,H_2)$ the threshold
for the property $G(n,p)\mr (H_1,H_2)$.

In this paper, we determine~$\hat p(H_1,H_2)$ when $H_2$ is a forest,
explicitly in most cases and implicitly in all the remaining cases.
Before we state our results, we recall the definition of some
graph densities.  For a graph~$H$, let~$v(H)$ and~$e(H)$ be the number
of vertices and edges in~$H$.
The \textit{maximum $2$-density} of~$H$, denoted by $m_2(H)$, is
defined as $m_2(K_1)=m_2(2K_1)=0$ and $m_2(K_2) = 1/2$ for graphs
with at most two vertices; if $v(H)\geq 3$, then we define it as
\[
  m_2(H)=\max\left\{\frac{e(J)-1}{v(J)-2} \;:\; J\subset
    H,\;v(J)\geq 3\right\}.
\]
The \emph{maximum density} of a graph~$H$ with~$v(H)\geq1$ is denoted by
\begin{equation*}
  m(H)=\max\left\{\,\frac{e(J)}{v(J)} \;:\; J\subset
    H,\;v(J)\geq1\,\right\}.
\end{equation*}

We now discuss the threshold~$\hat p(H_1,H_2)$ of the property $G\mr
(H_1,H_2)$, \emph{always assuming that~$H_2$ is a forest} (we briefly
discuss the other case in which~$\rc(H_1,H_2)$ exists, namely
when~$H_1$ is a star and~$H_2$ is not a forest, at the end of this
introduction).  Below, we introduce the main cases, which will be
proved in Theorem~\ref{thm:main}. The remaining cases will be
discussed after stating Theorem~\ref{thm:main}, and will be summarised
in Proposition~\ref{prop:main}.

From the celebrated result of R\"odl and
Ruci\'nski~\cite{rodl93:_lower_ramsey, RR}, we know that if $H_1$ is not a star forest, then
for $p\ll n^{-1/m_2(H_1)}$ with high probability there is a colouring $\chi$ of the
edges of~$G(n,p)$ with two colours containing no monochromatic copy
of~$H_1$.  Clearly, for~$e(H_2)\geq3$, there
is no rainbow copy of $H_2$ in $\chi$ and hence~$\hat p(H_1,H_2)=\Omega(
n^{-1/m_2(H_1)})$.
A matching upper bound for the threshold holds in this case (see
Theorem~\ref{thm:main}~\ref{item:thm_i}), but we do not
deduce it from the $1$-statement in~\cite{RR}.

Now let $H_1$ be a star forest\footnote{Note that a matching is a star
  forest.}.  We call a disjoint union of edges and cherries a
\emph{short forest}.  Furthermore, we call a star forest that has at least two
components a \emph{constellation}.  If~$H_1$ is a constellation
that is not a matching and $H_2$ is not a short forest, then we prove
that, again, $\hat p(H_1,H_2) = n^{-1/m_2(H_1)}$ (see
Theorem~\ref{thm:main}~\ref{item:thm_i}).
On the other hand, if $H_1$ is a star, or else $H_1$ is a
constellation and $H_2$ is a short forest, then the threshold depends
on a parameter $\mF(H_1, H_2)$, defined by
\begin{equation*}
  \mF(H_1,H_2) = \inf\{m(F)\,:\, \text{$F$ is a
    forest and }F\mr(H_1,H_2)\}.
\end{equation*}
Our main theorem is as follows. Its two parts will be proven in
Sections~\ref{sec:m2}~and~\ref{sec:mf}.

\begin{theorem}\label{thm:main} Let $H_1$ be a graph and $H_2$ be a
forest with $e(H_1) \geq 2$ and $e(H_2)\geq 3$.
  \begin{enumerate}[label=\rmlabel]
  \item\label{item:thm_i} If $H_1$ is not a star forest, or else~$H_1$ is a constellation
    with $\Delta(H_1) \geq 2$ and $H_2$ is not a short forest, then\label{eq:item1}
    \begin{equation}
      \label{eq:main2} \hat p(H_1,H_2) = n^{-1/m_2(H_1)}.
    \end{equation}

  \item\label{item:thm_ii} If $H_1$ is a star, or else $H_1$ is a
    constellation and $H_2$ is a short forest, then $\mF(H_1,H_2) <
    1$ and\label{eq:item2}
     \begin{equation}
      \label{eq:main1} \hat p(H_1,H_2) = n^{-1/\mF(H_1,H_2)}.
    \end{equation}
  \end{enumerate}
\end{theorem}

We remark that, for graphs $H_1$ and $H_2$ as in the statement of
Theorem~\ref{thm:main}~\ref{eq:item2}, the infimum in the definition
of $\mF(H_1,H_2)$ is attained, that is, for such graphs $H_1$ and
$H_2$ we show that there is a forest~$F$ with $m(F) = \mF(H_1,H_2)$
(see Corollary~\ref{cor:well}).

Recall that we are considering the case in which~$H_2$ is a forest.
Even though Theorem~\ref{thm:main} deals with the most interesting
cases under this restriction, it does not cover all cases.
If one of~$H_1$ or~$H_2$ has only one edge,
then trivially~$\hat p(H_1,H_2)=n^{-2}$.
The remaining cases are covered by Proposition~\ref{prop:main},
which is stated below and will be proved in Appendix~\ref{sec:smallcases}.

\begin{proposition}\label{prop:main}
  Let $H_1$ be a graph and $H_2$ be a forest with
  $\min\{e(H_1),\,e(H_2)\}\geq 2$.
  \begin{enumerate}[label=\rmlabel]
     \item\label{item:prop_i} If $H_1$ is a matching and $H_2$ is not a short forest, then\label{eq:item3}
     \begin{equation}
      \label{eq:mainprop1}
      \hat p(H_1,H_2) = n^{-1}.
    \end{equation}

     \item\label{item:prop_ii} If $H_1$ is not a forest and $H_2$ is a
       cherry, or else~$H_1$ is arbitrary and~$H_2$ is a $2$-edge
       matching, then\label{eq:item4}
     \begin{equation}
      \label{eq:mainprop2}
      \hat p(H_1,H_2) = n^{-1/m(H_1)}.
    \end{equation}

  \item\label{item:prop_iii} If $H_1$ is a forest with $k$
    non-isolated vertices and $H_2$ is a cherry, then\label{eq:item5}
     \begin{equation}
      \label{eq:mainprop3}
      \hat p(H_1,H_2) = n^{-k/(k-1)}.
    \end{equation}
\end{enumerate}
\end{proposition}

We observe that when $H_2$ is a forest and~$\min\{e(H_1),\,e(H_2)\}\geq2$,
Theorem~\ref{thm:main} and Proposition~\ref{prop:main} cover all
possibilities.  Indeed, the case $e(H_2) = 2$ is covered by
Proposition~\ref{prop:main}~\ref{eq:item4} and~\ref{eq:item5}.  For
$e(H_2)\geq 3$, note first that the case in which~$H_1$ is not a star
forest is covered in Theorem~\ref{thm:main}~\ref{eq:item1}.  On the
other hand, if~$H_1$ is a star forest, then~$H_1$ is either a star, a
matching or a constellation that is not a matching.
Theorem~\ref{thm:main}~\ref{eq:item2} covers the case in which~$H_1$
is a star.  Theorem~\ref{thm:main}~\ref{eq:item2} and
Proposition~\ref{prop:main}~\ref{eq:item3} cover the case in
which~$H_1$ is a matching\footnote{Note that a non-trivial matching is
  a constellation.}.  Theorem~\ref{thm:main}~\ref{eq:item1}
and~\ref{eq:item2} cover the case in which~$H_1$ is a constellation
that is not a matching.  We have thus covered all cases
when~$H_2$ is a forest.

The problem of determining the threshold for $G(n,p)\mr (H_1,H_2)$
when~$H_2$ is not a forest is still open.  Recall that the only
meaningful case of this problem when~$H_2$ is not a forest is
when~$H_1$ is a star.  When~$H_1$ is a star, the problem of
determining~$\hat p(H_1,H_2)$ generalises the problem of determining
the threshold for the so-called \emph{anti-Ramsey property}
of~$G(n,p)$ for the graph~$H_2$, which states that every proper edge
colouring of~$G(n,p)$ should contain a rainbow copy of~$H_2$.
Clearly, the case in which~$H_1=K_{1,2}$ is this anti-Ramsey property
exactly (proper edge colourings coincide with colourings that avoid
monochromatic~$K_{1,2}$).  The more general case in
which~$H_1=K_{1,r}$ is investigated in~\cite{KoKoMo14}, where it is
proved that for every fixed $r\geq2$ and any~$H_2$, with high
probability we have $G(n,p) \mr (K_{1,r}, H_2)$ whenever
$p \gg n^{-1/m_2(H_2)}$.  However, $n^{-1/m_2(H_2)}$~turns out not to
be the threshold for some graphs, as shown in~\cite{KoKoMo18} (for related results,
see~\cite{BaCaMoPa21+, KoMoPaSc21+, NePeSkSt17}).

This paper is organised as follows. In Section~\ref{sec:prelim} we
provide some results about random graphs that will be useful in the proofs of
Theorem~\ref{thm:main} and Proposition~\ref{prop:main}.
The items~\ref{eq:item1} and~\ref{eq:item2} of
Theorem~\ref{thm:main} are proved, respectively, in
Sections~\ref{sec:m2} and~\ref{sec:mf}.
Some open problems are discussed in Section~\ref{sec:open}.
 In Appendix~\ref{sec:smallcases} we give a short proof of
  Proposition~\ref{prop:main} and in Appendix~\ref{sec:mfbounds} we estimate the parameter $\mF(K_{1,3},H_2)$ when $H_2$ is a complete binary tree or a path.

\section{Random graphs}\label{sec:prelim}

Given a graph $G$, a colouring $\chi \colon E(G) \to \mathbb{N}$, a
vertex $v \in V(G)$ and a set $X \subset V(G)$, let
the \emph{colour-degree} of $v$ in $X$ be given by $d_\chi(v, X) =
\left|\left\{ \chi(e) : u\in e\text{ and }e \setminus \{v\} \subset X \right\}\right|$.
We write simply $d_\chi(v)$ for $d_\chi(v,V(G))$.
The following definition plays an important rôle in our proof.

\begin{property} Let $H$ be a graph, $r\geq 2$ be an
integer, and $0 < b \leq 1$. A graph $G$ satisfies property $\Q(b,r,H)$ if
every edge colouring $\chi$ of $G$ with no monochromatic copy of $H$
is such that every subset $X\subset V(G)$ with $|X| \geq bn$ contains a
vertex $v$ with $d_\chi(v, X) > r$.
\end{property}

The aim of this section is to prove the following result, which is
important in the proof of the $1$-statement of Theorem~\ref{thm:main}~\ref{eq:item1}.

\begin{theorem}\label{thm:property_Q} Let $H$ be a connected graph, $r
  \geq 2$ be an integer and $0 < b \leq 1$. If $p \gg n^{-1/m_2(H)}$, then $G(n,p)$
  satisfies property~$\Q(b, r, H)$ with high probability.
\end{theorem}

We shall need the following well-known result of Bollobás.

\begin{theorem}[\cite{Bo81}]\label{thm:BB} Let $H$ be an arbitrary
graph with at least one edge. Then the threshold for $H$ to be a
subgraph of $G(n,p)$ is $n^{-1/m(H)}$.
\end{theorem}

We write $G\to (H)_r$ for the property that in every $r$-colouring of
the edges of $G$ there is a monochromatic copy of $H$. When
dealing with a graph $H_1$ that is not a star forest, we use the
following consequence of a celebrated result of R\"odl and
Ruci\'nski to obtain the $0$-statement for the property $G(n,p)\mr
(H_1,H_2)$.

\begin{theorem}[\cite{rodl93:_lower_ramsey, RR}]\label{RR-original} For
every integer $r\geq 2$ and every graph $H$ which is not a star
forest, the threshold for the property $G(n,p)\to (H)_r$ is $n^{-1/m_2(H)}$.
\end{theorem}

We also need the following strengthening of the $1$-statement of
Theorem~\ref{RR-original}.
\begin{theorem}[Theorem 3 of~\cite{RR}]\label{RR-strength} Let $H$ be
a graph with at least one edge and let $r \geq 2$ be an integer. There
exist constants $n_0$, $b$ and $C$ such that if $n \geq n_0$ and $p \geq
Cn^{-1/m_2(H)}$, then
  \[ \mathbb{P}\left(G(n,p) \to (H)_r\right) \geq 1 - \exp\left(-b n^2
p\right). \]
\end{theorem}

By using the union bound and Theorem~\ref{RR-strength}, we obtain the following corollary.

\begin{corollary}\label{cor:probabilistic} Let $r \geq 2$,
  $\varepsilon > 0$, and $H$ be a graph with $\Delta(H) \geq 2$.
  If $p \gg n^{-1/m_2(H)}$, then the following holds with high probability for $G=G(n,p)$.
  For every edge colouring $\chi$ of $G$ with no monochromatic copy of~$H$ and every $X \subset V(G)$ of size $|X| \geq \varepsilon n$, the graph $G[X]$ is coloured with more than $r$ colours under $\chi$.
\end{corollary}

\begin{proof} Let $n_0 \in \mathbb{N}$ and $b > 0$ be given by applying Theorem~\ref{RR-strength}, let $p \gg n^{-1/m_2(H)}$ and note that since $H$ is not a matching, we have $p\gg n^{-1}$.

  If $n$ is larger than $n_0/\varepsilon$, the probability that a subset $X
  \subset V(G)$ of size $\varepsilon n$ is such that $G[X]$ admits an $r$-colouring with no monochromatic copy of $H$ is at most $e^{-bn^2 p}$, by Theorem~\ref{RR-strength}.
  By the union bound, the desired conclusion fails with probability at most $2^n \cdot e^{-bn^2p}=o(1)$.
\end{proof}

The next step is to strengthen Corollary~\ref{cor:probabilistic} by making
the restriction on the number of colours hold locally. We say a
colouring $\chi$ of a graph $F$ is \emph{$r$-local} if $d_\chi(v) \leq
r$ for every $v \in V(F)$. The following lemma is due to Ruci\'{n}ski and Truszczy\'{n}ski \cite{RT}.

\begin{lemma}[Lemma 2 of~\cite{RT}]\label{lem:r_local} Let $F$ and $H$
be graphs and suppose $H$ is connected. If there is an $r$-local
colouring $\chi$ of $F$ with no monochromatic copy of~$H$, then there
exists $W \subset V(F)$ of size $|W| \geq (r!/r^{r})\cdot v(F)$ and an
$r$-colouring $\chi'$ of the edges of $F[W]$ with no monochromatic
copy of~$H$.
\end{lemma}

Combining Corollary~\ref{cor:probabilistic} and
Lemma~\ref{lem:r_local}, we prove Theorem~\ref{thm:property_Q}.

\begin{proof}[Proof of Theorem~\ref{thm:property_Q}]
  Let a connected graph $H$, an integer $r \geq 2$ and $b > 0$ be given.
  Let us consider $G=G(n,p)$ with $p\gg n^{-1/m_2(H)}$.
  If $e(H) \leq 1$, then $m_2(H) = m(H)$ and therefore $H \subset G$ with high probability.
  In this case, the conclusion holds vacuously, because every copy of $H$ is trivially monochromatic.
  We may therefore assume that $\Delta(H) \geq 2$, since $H$ is connected.

    Take $\varepsilon = br!/r^r$, let $\chi$ be an edge colouring of $G$ with no monochromatic copy of~$H$ and consider a set $X
    \subset V(G)$ of size $|X| \geq b n$.

    Suppose for a contradiction that $\chi$ restricted to $G[X]$ is
    $r$-local. By Lemma~\ref{lem:r_local} applied with $F = G[X]$ and
    $H$, there exists a set $W\subset X$ with $|W| \geq (r!/r^r)|X|
    \geq \varepsilon n$ and an $r$-colouring $\chi'$ of the edges of
    $G[W]$ with no monochromatic copy of~$H$. But, from
    Corollary~\ref{cor:probabilistic}, we know that with high
    probability there is no such $r$-colouring, which concludes the proof.
\end{proof}

\section{Threshold at $n^{-1/m_2(H_1)}$}
\label{sec:m2}

Let $H_1$ be a graph and $H_2$ be a forest with $e(H_1)\geq 2$ and $e(H_2)\geq 3$.
In this section we prove Theorem~\ref{thm:main}~\ref{eq:item1}, which states that $\hat p(H_1,H_2) = n^{-1/m_2(H_1)}$ when $H_1$ is not a star forest, and also when $H_1$ is a constellation with $\Delta(H_2)\geq 2$ and $H_2$ is not a short forest.

\subsection{$1$-statement}

In this section it will be useful to assume that $H_1$ is connected.
This assumption is justified by Proposition~\ref{prop:density}
below, which when iterated implies that every disconnected graph with
at least one edge and components $J_1,\ldots,J_k$ is a spanning subgraph of a
connected graph $H_1$ with $m_2(H_1)=\max\{1, m_2(J_1), \ldots, m_2(J_k)\}$.

\begin{proposition}
  \label{prop:density}
Let $G_1$ and $G_2$ be connected graphs on disjoint vertex sets such that $m_2(G_1) \geq
m_2(G_2)$ and $e(G_1)\geq 1$. If $G$ is a graph obtained by adding a single edge between $G_1$
and $G_2$, then $m_2(G) = \max\{m_2(G_1), 1\}$.
\end{proposition}

The proof of Proposition~\ref{prop:density} is straightforward and is
postponed to Appendix~\ref{sec:smallcases}.

To prove the $1$-statement of Theorem~\ref{thm:main}~\ref{eq:item1},
we shall use Theorem~\ref{thm:property_Q} to prove that, for every
connected graph $H$, and every fixed tree $T$, we have $G(n,p) \mr (H, T)$ with high probability as long as $p \gg n^{-1/m_2(H)}$.
For this purpose, we will consider the complete $d$-ary tree of height~$h$, for general $h$ and $d$, which we denote by~$T(d,h)$.

\begin{theorem}\label{thm:trees} Let $H$ be a connected graph and let
  $d$ and $h$ be positive integers. If $p \gg n^{-1/m_2(H)}$, then
  $G(n,p) \mr (H, T(d,h))$ with high probability.
\end{theorem}

Since Theorem~\ref{thm:property_Q} states that $G(n,p)$ satisfies
property $\Q(b, r, H)$ with high probability when $p \gg
n^{-1/m_2(H)}$, Theorem~\ref{thm:trees} follows directly from the
following deterministic lemma.

\begin{lemma}\label{lem:trees} Let $H$ be a connected graph. For all
  positive integers $d$ and $h$, there exist positive reals $b$, $c < 1$ and an
  integer $r \geq 1$ with the following property. If a graph $G$ satisfies
  $\Q(b, r, H)$, then in any edge colouring of $G$ either there is a
  monochromatic copy of $H$ or there are at least $\lfloor c \cdot
  v(G) \rfloor$
  vertex-disjoint copies of $T(d,h)$, each of them
  rainbow.

\end{lemma}

\begin{proof} Let $H$ be a connected graph. Our proof is by induction
  on $h$.  For $h = 1$, we take $b = 1/2$, $c = (2(d+1))^{-1}$ and
  $r=d-1$. Let $\chi$ be an edge colouring of an $n$-vertex graph $G$
  that satisfies $\Q(b, r, H)$ and suppose that there are no
  monochromatic copy of $H$ under~$\chi$.  Recall that $\Q(b, r, H)$
  implies that every subset $X \subset V(G)$ with $|X| \geq bn$
  contains a vertex incident to edges coloured with more than $r$
  colours in $X$.  Therefore, since $h = 1$, from the choice of $b$,
  $c$ and $r$, the definition of $\Q(b, r, H)$ allows us to
  iteratively find rainbow copies of $T(d, 1)$ until we have used more
  than $n/2$ vertices of $G$.  This procedure therefore finds $\lfloor
  n(2(d+1))^{-1}\rfloor\geq\lfloor cn\rfloor$
  disjoint rainbow copies of $T(d,1)$, as claimed.

  We now show that the result holds for $h > 1$. Let $b'$, $c'$ and
$r'$ be obtained by applying the base case $h' = 1$ with $d' :=
2d^{h}$. Thus, we have $b'=1/2$, $c'= (2(d'+1))^{-1}$ and $r'=d'-1$. Also, let $b''$, $c''$ and $r''$ be obtained by applying the
induction step with $h'' := h-1$ and $d'' := d$. We will show below that
the conclusion of the lemma holds for $b = b''c'/2$, $r = \max\{r',
r''\}$ and $c = c' c''/2$.

  Let $G$ be a graph satisfying $\Q(b, r, H)$, and observe that we may
assume $cn \geq 1$ because the conclusion of the lemma is
vacuous otherwise. As before, suppose $E(G)$ is coloured with no
monochromatic copy of $H$ under $\chi$. Since $b' > b$ and $r'\leq r$, the graph $G$
satisfies $\Q(b', r', H)$ and, from the base case, $G$ contains a
family $\mathcal{L}$ of at least $\lfloor c'n\rfloor\geq c'n/2$ rainbow vertex disjoint copies of $T(2d^{h}, 1)$. Let $X \subset
V(G)$ be the set of roots of such trees. Observe that, since
$b''|X|\geq b''c'n/2 = bn$ and $r'' \leq r$, the graph $G[X]$ satisfies
$\Q(b'', r'', H)$. Therefore, by the induction hypothesis, $G[X]$
contains a family $\mathcal{T}$ of $\lfloor  c'' \cdot v(G[X])\rfloor
\geq \lfloor cn\rfloor$ vertex disjoint rainbow rooted
copies of $T(d'',h'')$, i.e., copies of $T(d,h-1)$.

  Notice that, by the definition of $X$, each leaf $v$ of a tree $T \in
\mathcal{T}$ is the root of a tree $L_v \in \mathcal{L}$. Since $T$
has at most $d^{h}$ edges, there are $d^{h}$ edges in each $L_v$ whose
colours do not appear in $T$. A greedy procedure can then be used to
extend $T$ to a rainbow tree of height $h$, concluding the induction step
and the proof of the lemma.
\end{proof}

\subsection{$0$-statement}\label{subsec:nonshort_zerostmt}

Let $H_2$ be a forest with $e(H_2)\geq 3$. In this section we prove
the $0$-statement of Theorem~\ref{thm:main}~\ref{eq:item1}, i.e., if $p\ll n^{-1/m_2(H_1)}$, then the following holds
when $H_1$ is not a star forest, and also when $H_1$ is a
constellation with $\Delta(H_1)\geq 2$ and $H_2$ is not a short forest: with high
probability, there is
an edge colouring of $G(n,p)$ with no monochromatic copy of
$H_1$ and no rainbow copy of $H_2$.

We start by noticing that if $H_1$ is not a star forest, then the
result follows directly from the $0$-statement of the
R\"odl--Ruci\'nski theorem (Theorem~\ref{RR-original}) with $r = 2$ (recall that $H_2$ has at least three edges).
Thus, we may and shall assume that $H_1$ is a constellation with $\Delta(H_1)\geq 2$ and $H_2$ is not a short forest.
In fact, the proof we present below also works when $H_1$ is a matching.

Since $H_1$ is a forest with $v(H_1)\geq 3$, we have $m_2(H_1)=1$.
For $p\ll n^{-1/m_2(H_1)} = n^{-1}$, with high probability $G(n,p)$ is a forest.
Therefore, to obtain the aimed $0$-statement and finish the proof
of Theorem~\ref{thm:main}~\ref{eq:item1}, it is enough to prove Proposition~\ref{prop:provideColouring} below.
\begin{proposition}
\label{prop:provideColouring}
Let $H_1$ be a constellation and $H_2$ be a forest with $e(H_2)\geq 3$
that is not a short forest. Then, for any forest $F$ there is an edge
colouring of $F$ with no monochromatic copy of $H_1$ and no rainbow
copy of $H_2$.
\end{proposition}

\begin{proof}
  Let $H_1$ be a constellation and $H_2$ be a forest with $e(H_2)\geq 3$ that is not a short forest.
Let $F$ be an arbitrary $n$-vertex forest composed by trees $T_1,
\ldots, T_k$, rooted at arbitrary vertices.
Define $(v_1,\ldots, v_n)$ to be an ordering of $V(F)$ such that
the vertex depths are non-decreasing (that is, $\operatorname{depth}(v_i) > \operatorname{depth}(v_j)$
implies $i > j$).
In what follows, we construct a colouring $\chi\colon E(F)\to\mathbb{N}$ that contains no monochromatic copy of $H_1$ and no rainbow copy of $H_2$.

Since $H_2$ is not a short forest, either $\Delta(H_2)\geq 3$ or
$H_2$ contains a path with three edges. If $\Delta(H_2)\geq 3$, then
put $\chi(v_iv_j) = \min\{i, j\}$ for every edge $v_iv_j$.
The colouring $\chi$ clearly has no monochromatic copy of $H_1$ as there are no monochromatic vertex-disjoint stars with the same colour.
Also, since every vertex of $F$ is incident to edges coloured with at most two colours and $\Delta(H_2)\geq 3$, there is no rainbow copy of $H_2$ in $F$.
Finally, if $H_2$ contains a path with $3$ edges, then we colour every edge $e\in E(F)$ by setting $\chi(e) = i$, where $v_i$ is the unique vertex of $e$ at odd depth in the tree containing $e$.
As before, since there are no monochromatic vertex-disjoint stars with the same colour, $F$ contains no monochromatic copy of $H_1$ under $\chi$.
Furthermore, there is no rainbow path $v_{i_0}v_{i_1}v_{i_2}v_{i_3}$,
since either $v_{i_1}$ or $v_{i_2}$ would have odd depths and
therefore this path would have two edges with the same colour.
\end{proof}

\section{Threshold at $n^{-1/\mF(H_1,H_2)}$}
\label{sec:mf}

Let $H_1$ be a graph with $e(H_1)\geq 2$ and let $H_2$ be a forest
with $e(H_2)\geq 3$ and recall that, by definition, $\mF(H_1,H_2) = \inf\{m(F)\colon \text{$F$ is a
    forest and }F\mr(H_1,H_2)\}$.

Here we prove Theorem~\ref{thm:main}~\ref{eq:item2}, showing the threshold $n^{-1/\mF(H_1,H_2)}$ for $G(n,p)\mr(H_1,H_2)$ when $H_1$ is a star, and also when $H_1$ is a constellation and $H_2$ is a short forest.

Propositions~\ref{prop:star} and~\ref{prop:forest} below show the existence of a forest $F$ such that $F\mr(H_1,H_2)$ (see Corollary~\ref{cor:well}), which imply that the parameter $\mF(H_1,H_2)$ is well defined for such particular graphs $H_1$ and $H_2$.
This fact guarantees that the threshold for $G(n,p)\mr(H_1,H_2)$ is given by $n^{-1/\mF(H_1,H_2)}$.
In fact, since $F\mr(H_1,H_2)$, the $1$-statement follows from the fact that $F \subset G(n,p)$ with high probability (Theorem~\ref{thm:BB}) and the $0$-statement follows from a simple argument and the definition of $\mF(H_1,H_2)$ (details are given in the end of this section).

\begin{proposition}\label{prop:star} If $H_1$ is a star and $H_2$ is a
forest, then there exists a tree $T$ such that $T\mr(H_1,H_2)$.
\end{proposition}

\begin{proof}
 Let $H_1$ be a star with $s$ edges.
  Consider a tree $H_2'$ that contains $H_2$ as a spanning subgraph, i.e., $V(H_2') = V(H_2)$ and $E(H_2)
  \subset E(H_2')$, rooted at some arbitrary vertex $v$.
  Let $e(H_2') = \ell$, and let $h$ be the height of $H_2'$.
  We will show that a complete $((s-1)(\ell-1)+1)$-ary tree $T$ of height $h$ satisfies $T\mr(H_1,H_2')$, which implies $T\mr(H_1,H_2)$.

  Note that in any edge colouring avoiding a monochromatic copy of $H_1$
there are at most $s-1$ edges with any given colour at each vertex of
$T$. Thus, the edges from every non-leaf of $T$ to its children
must be coloured with at least $\ell$ different colours. Therefore, a
greedy embedding that starts by assigning the root $v$ of $H_2'$ to the root of $T$ and
always chooses edges of previously unused colours, level by level, will
succeed in finding a rainbow copy of $H_2'$ in any colouring of $T$ that
avoids a monochromatic copy of $H_1$.
\end{proof}

In the next proposition we consider the case where $H_1$ is a
constellation and $H_2$ is a short forest. Given a graph
$G$ with edges coloured by $\chi$ and $v\in V(G)$, recall that
$d_{\chi}(v)$ denotes the number of colours used on edges incident to $v$.

\begin{proposition}\label{prop:forest}
  If $H_1$ is a constellation and $H_2$ is a short forest, then there exists a tree $T$ such that $T\mr(H_1,H_2)$.
\end{proposition}

\begin{proof}
  Extending $H_1$ and $H_2$ if necessary, we may assume that $H_1$ has $s$ stars with $s$ edges each and that $H_2$ is composed by $s$ cherries.

  Let $d=6s^3+7s^2$.
  We prove that any complete $d$-ary tree $T$ of height $3$ satisfies $T\mr(H_1,H_2)$.
  Let $\chi$ be an arbitrary edge colouring of such a tree $T$ with no monochromatic copy of $H_1$, and let $X \subset V(T)$ be the set of vertices of $T$ with colour-degree at least $3s$.
  We may assume that $|X| < 3s$, for otherwise one can find a rainbow copy of the forest of cherries $H_2$ greedily.
  Assuming this, the following claim holds.

\begin{claim}
  Let $A$ be a set of internal vertices of $T$ such that $|A| \geq 2s^2
  + 3s$.
  One may find $2s$ vertex-disjoint monochromatic stars of different
  colours centred at vertices of $A$, each of them of size $2s^2+3s$.
\end{claim}

\begin{proof}
  Removing vertices from $A$ if necessary, we obtain a set $S$ such
  that $S \cap X = \emptyset$ and $|S| \geq 2s^2 + 1$.  Every $v \in
  S$ has $d_\chi(v) \leq 3s-1$, and therefore each $v \in S$ is the
  centre of a monochromatic star with $d/(3s-1) > 2s^2 + 3s$ edges (to
  children of $v$).
  These stars span at least $(2s^2 + 1)/(s-1) > 2s$ colours,
  since there is no monochromatic copy of $H_1$ in $T$.
\end{proof}

To finish the proof of the proposition, let $r$ be the root of $T$ and apply the claim to $N(r)$ to obtain a collection $\mathcal{S}$ of $2s$ monochromatic stars of different colours.
We will construct a rainbow family of cherries $C_1, \ldots, C_s$ inductively.
Having constructed cherries $C_1, \ldots, C_{i-1}$, for $i < s$,
choose a star $S_i \in \mathcal{S}$ whose colour is \emph{new} (that
is, it does not appear in any of the previous $i-1$ cherries).
Let $v_i$ be the centre vertex of this star, and $W_i$ be the set of its leaves.
Applying the claim to $W_i$, we obtain a collection $\mathcal{S}_i$ of $2s$ monochromatic stars centred at vertices of $W_i$.
We may then choose a star $Z_i \in \mathcal{S}_i$ such that its colour is new and differs from $S_i$'s colour.
Any cherry containing $v_i$ and any edge of $Z_i$ is then a valid
choice for $C_i$. Thus, we found a rainbow copy of $H_2$, as desired.
\end{proof}

\begin{corollary}\label{cor:well}
  Let $H_1$ be a star and $H_2$ be a forest, or else let $H_1$ be a constellation and $H_2$ be a short forest. There is a
  forest $F$ such that $F\mr(H_1,H_2)$ and for any forest $F'$ with
  $m(F')<m(F)$ we do not have $F'\mr(H_1,H_2)$.
\end{corollary}

\begin{proof}
  Let $v(H_1,H_2) = \min\{k\in \mathbb{N}\colon $ there is a forest
  $F$ with components of size at most $k$ such that $F\mr(H_1,H_2)\}$.
  In view of Propositions~\ref{prop:star}
  and~\ref{prop:forest}, the parameter $v(H_1,H_2)$ is well defined.
  By definition of $v(H_1, H_2)$, there exists a forest $F$ with
  components of size at most $v(H_1,H_2)$ such that $F\mr(H_1,H_2)$,
  and $F$ satisfies the desired conclusion.
\end{proof}

We now prove the main result of this section.

\begin{proof}[Proof of Theorem~\ref{thm:main}~\ref{eq:item2}]
  Let $H_1$ be a graph and $H_2$ be a forest with $e(H_1)\geq 2$ and $e(H_2)\geq 3$.
  Furthermore, let $H_1$ be a star and $H_2$ be a forest; or let $H_1$ be a constellation and $H_2$ be a short forest.
  From Corollary~\ref{cor:well}, there is a forest $F$ such that $m(F)
  = \mF(H_1,H_2)$. Let us denote by $k$ the size of the largest
  component of $F$, and observe that $m(F)=(k-1)/k$.

  Let $p\ll n^{-1/m(F)} \ll n^{-1}$ and let $G=G(n,p)$.
  Then, $G$ is a forest with high probability, and the expected number of copies of trees with $k$ vertices in $G$ tends to $0$ as $n$ tends to infinity.
  Therefore, from Markov's inequality, we conclude that every component of
  $G$ has fewer than $k$ vertices with high probability, and hence
  $m(G) < m(F)$. The infimum in the definition
  of $\mF(H_1,H_2)$ then implies the existence of a colouring of
  $G(n,p)$ containing no monochromatic copy of $H_1$ and no rainbow
  copy of $H_2$ with high probability.

  For $p \gg n^{-1/\mF(H_1, H_2)}$, the random graph $G(n,p)$ contains a copy of $F$ with high probability (Theorem~\ref{thm:BB}).
  Therefore, any colouring of $G(n,p)$ contains either a monochromatic copy of $H_1$ or a rainbow copy of $H_2$.
  This concludes the proof.
\end{proof}

\section{Open problems}
\label{sec:open}

Let $H_1$ be a graph and $H_2$ be a forest such that $e(H_1)\geq 2$
and $e(H_2)\geq 3$.
If $H_1$ is a star, then the proof of Proposition~\ref{prop:star} provides an
upper bound for $\mF(H_1,H_2)$ and the proof of Proposition~\ref{prop:forest}
gives an upper bound in the case where $H_1$ is a constellation and $H_2$ is a short forest.
One natural question is to ask for the exact value of $\mF(H_1,H_2)$
for these pairs of graphs $(H_1,H_2)$.
As a first step, in Appendix~\ref{sec:mfbounds} we estimate  $\mF(K_{1,3},H_2)$ when $H_2$ is a complete binary tree or a path.

We determine (apart from calculating the exact value of $\mF(H_1,H_2)$) the threshold
$\hat p(H_1,H_2)$ for $G(n,p)\mr (H_1,H_2)$ when $H_2$ is a forest.
The most relevant open question related to this problem is to
determine $\hat p(H_1,H_2)$ when $H_1$ is a star and $H_2$ is not a
forest, which is a generalisation of the so called anti-Ramsey problem (see~\cite{BaCaMoPa21+, KoKoMo14,
  KoKoMo18, KoMoPaSc21+, NePeSkSt17}).

\section{Acknowledgements}

The authors are grateful to Antonio K. B. Fernandes, Hugo
M. Vicente and Uriel A. S. Martínez for helpful discussions.

\appendix

\section{Proofs of Propositions~\ref{prop:main} and~\ref{prop:density}}
\label{sec:smallcases}
We start this short appendix by proving Proposition~\ref{prop:main},
which, in view of Theorem~\ref{thm:main}, completes the scenario when
$H_2$ is a forest.

\begin{proof}[Proof of Proposition~\ref{prop:main}]
  Recall that, in all cases, $H_1$ is a graph with $e(H_1)\geq 2$ and
  $H_2$ is a forest with $e(H_2)\geq 2$.

  We first prove Proposition~\ref{prop:main}~\ref{item:prop_i}, i.e.,
  we prove that if $H_1$ is a matching and $H_2$ is not a short forest, then $\hat p(H_1,H_2) = n^{-1}$.
  Indeed, for $p \ll n^{-1}$ we know that with high probability $G(n,p)$ is a forest and Proposition~\ref{prop:provideColouring} provides a colouring of any $n$-vertex forest with no monochromatic copy of $H_1$ and no rainbow copy of $H_2$.
  For $p \gg n^{-1}$, we just apply the $1$-statement of
  Theorem~\ref{thm:main}~\ref{eq:item1} to any tree containing
  $H_1$. This completes the proof of Proposition~\ref{prop:main}~\ref{eq:item3}.

  To prove Proposition~\ref{prop:main}~\ref{eq:item4}, we must
  consider the case where $H_1$ is an arbitrary graph and $H_2$ is a $2$-edge
  matching, or else $H_1$ is not a forest and $H_2$ is a cherry.
  First let $H_1$ be an arbitrary graph (with $e(H_1)\geq 2$) and let $p\ll n^{-1/m(H_1)}$.
  Note that from Theorem~\ref{thm:BB}, with high probability $G(n,p)$ contains no copy of $H_1$.
  If $H_2$ is a cherry or a $2$-edge matching, then any colouring with a single colour contains no rainbow copy of $H_2$.
  This finishes the proof of the $0$-statement for such graphs $H_1$ and $H_2$.

  For the $1$-statement of Proposition~\ref{prop:main}~\ref{eq:item4}, let $p \gg n^{-1/m(H_1)}$ and note that, again from Theorem~\ref{thm:BB}, with high probability $G(n,p)$ contains a copy of $H_1$.
  Since any colouring with a single colour contains a monochromatic copy of $H_1$, we assume that the edges of $G(n,p)$ are coloured with at least~$2$ colours.
  If $H_2$ is a $2$-edge matching, then for any two edges $e$ and $f$
  coloured with different colours, any edge disjoint from $e$ and $f$
  will complete a rainbow matching with one of $e$ or $f$, giving the desired rainbow copy of~$H_2$.
  Assume now that $H_1$ is not a forest and $H_2$ is a cherry.
  For this case we expose $G(n,p)$ in two rounds, considering $p_1,
  p_2 \gg n^{-1/m(H_1)}$ and $G(n,p) = G(n,p_1) \cup G(n, p_2)$. Let
  $G=G(n,p)$, $G_1=G(n,p_1)$ and $G_2=G(n,p_2)$.
  Note that since $H_1$ is not a forest, we have $m(H_1)\geq 1$, which implies $p_1\gg n^{-1}$.
  Then, with high probability, $G_1$ contains a component with at least $cn$
  vertices, for some positive $c$. Let $X$ be the vertex set of such a
  component.
Now note that since $p_2 \gg n^{-1/m(H_1)}$,
there is a copy of $H_1$ in $G_2[X]$ with high probability.
Therefore, if $G[X]$ is monochromatic we obtain a monochromatic copy
of $H_1$. Otherwise, we obtain a rainbow cherry, which concludes the
proof of Proposition~\ref{prop:main}~\ref{eq:item4}.

  It remains to prove Proposition~\ref{prop:main}~\ref{eq:item5},
  in which $H_1$ is a forest and $H_2$ is a cherry.
  Without loss of generality, we may assume that $H_1$ has no isolated
  vertices. For
  the $0$-statement one can check that for $p\ll n^{-v(H_1)/(v(H_1)-1)}$,
  with high probability every component of $G(n,p)$ is a tree with fewer than
  $v(H_1)$ vertices. Then, we may colour every component in a
  monochromatic way, using different colours for different components,
  avoiding monochromatic copies of $H_1$ and rainbow copies of $H_2$.
  For the $1$-statement, let $T$ be a tree on $v(H_1)$ vertices
  containing $H_1$ as a subgraph. If $p \gg
  n^{-v(H_1)/(v(H_1)-1)}$, then there is a copy of $T$ in
  $G(n,p)$ with high probability (Theorem~\ref{thm:BB}). If this copy
  is monochromatic, we have the desired copy of $H_1$. Otherwise,
  there is a rainbow cherry in $G(n,p)$, as desired.
\end{proof}

We conclude this section by proving Proposition~\ref{prop:density}.

\begin{proof}[Proof of Proposition~\ref{prop:density}]
  If $m_2(G_1) < 1$, then $G_1 = K_2$ and $G_2$ has at most~$2$ vertices.
  Therefore, $G$ is a forest with at least three vertices, from where we conclude that $m_2(G) = 1$ and hence we may assume that $m_2(G_1) \geq 1$.
  Let $H \subset G$ be an arbitrary subgraph of $G$ with at least three vertices.
  Our aim is to prove that $(e(H)-1)/(v(H)-2) \leq m_2(G_1)$, which implies that $m_2(G) \leq m_2(G_1)$.
  For that, let $A_i = V(G_i) \cap V(H)$ for $1\leq i\leq 2$.
  Note that we may also assume that $A_1$ and $A_2$ are both nonempty,
  as otherwise $H$ would be a subgraph of $G_1$ or $G_2$, which
  implies that $(e(H)-1)/(v(H)-2)\leq m_2(G_1)$.

  If $|A_1|, |A_2| \geq 3$, then using that $(a+b)/(c+d) \leq \max\{a/c, b/d\}$ for any $a, b \geq 0$ and $c, d > 0$, we have
  \begin{align*}
    \frac{e(H) - 1}{v(H) - 2} &\leq \frac{(e(G[A_1])-1) +
    (e(G[A_2])-1) + 2}{(|A_1|-2) + (|A_2|-2) + 2}\\
                              &\leq \max\{m_2(G[A_1]),m_2(G[A_2]),1\}\\
                                &\leq \max\{m_2(G_1), 1\}.
  \end{align*}
  If $|A_1|, |A_2| \leq 2$, then the
graph $H$ has no cycles and therefore $m_2(H) \leq 1$.
Moreover, if $|A_i| \geq 3$ and $|A_{3-i}| \in \{1,2\}$ for some $1 \leq i
\leq 2$, then $e(H) - e(H[A_i]) \leq |A_{3-i}|$. Therefore, by a
similar argument as before,
\[\frac{e(H) - 1}{v(H) - 2} = \frac{(e(H[A_i]) - 1) + (e(H) -
    e(H[A_i]))}{(|A_i| - 2) + |A_{3-i}|} \leq \max\{m_2(G_i), 1\},\]
thereby finishing the proof.
\end{proof}

\section{Bounds for $\mF(K_{1,3},H)$}
\label{sec:mfbounds}

 Let $H$ be a forest and recall that, by definition,
\[
  \mF(K_{1,3},H)= \inf\{m(F)\colon \text{$F$ is a forest and }F\mr(K_{1,3},H)\}.
\]
For completeness, we sketch the result obtained in~\cite{CoFeMoVi21},
which provides bounds on $\mF(K_{1,3},H)$ when $H$ is a path or a
complete binary tree.  Since the proof for paths and complete binary
trees are similar, in this appendix we consider only the case where
$H$ is a complete binary tree. Our aim is to prove the following
result.
\begin{theorem}\label{teo:arvore_binaria}
  Let $H$ be the complete binary tree of height $h$. Then,
  \begin{enumerate}[label=\rmlabel]
    \item \label{lower-mF} if $T$ is the largest tree of a forest $F$
      with $F\mr(K_{1,3},H)$, then $v(T)\geq 2^{\binom{h-1}{2}}$, and
    \item \label{upper-mF} there is a rooted tree $G$ with $2^{(1+o(1))\binom{h}{2}}$ vertices such that $G\mr(K_{1,3},H)$, where the term $o(1)$ denotes a function that tends to zero as $h$ tends to infinity.

    \end{enumerate}
\end{theorem}

Observe that, for any forest $F$, we have $m(F) = 1 - v(T)^{-1}$,
where $T$ is the largest tree of $F$.  Therefore, by unfolding
the definition of $\mF$, Theorem~\ref{teo:arvore_binaria} implies
that   \[
    \mF(K_{1,3}, H) = 1 - 2^{-(1+o(1))\binom{h}{2}}.
  \]

In what follows, we denote by $(H,r)$ a rooted tree which consists of
an unrooted tree $H$ and a root $r$. Given a rooted tree $T = (H, r)$
and $v, w \in V(H)$, we say $w$ is a \emph{descendant} of $v$ if the
path connecting $r$ and $w$ contains $v$.  Let $D_T(v)$ denote the
number of descendants of $v$ (including $v$).

\begin{definition}[descendant colouring]\label{def:colouring_by_weight}
Given a rooted tree $T$, a \emph{descendant colouring} of $T$ is a function $\chi \colon E(T) \to \mathbb{N}$ with the following property: for any vertex $u$, if $v_1, \ldots, v_k$ are the children of $u$, then $\{\chi(uv_1), \ldots, \chi(uv_k)\} = \{1, \ldots, k\}$ and $\chi(uv_i) < \chi(uv_j)$ for any $1 \leq i,j \leq k$ such that $D_T(v_i) > D_T(v_j)$.
\end{definition}

To colour a tree $T$ with a descendant colouring, we sort the children
of each vertex $u$ in descending order of number of descendants (arbitrarily breaking ties) and colour the edges between $u$ and its children in that order using colours from $1$ to $k$.

The following observation follows directly by noting that, given a child $u$ of a vertex $v$ in a rooted tree $T$, the vertex $u$ has at least $\chi(uv)$ children with $D_T(v)$ or more descendants.

\begin{observation}\label{obs:weight_consequence}
    Let $T$ be a rooted tree and $\chi$ be a descendant colouring of
    $T$. If $v$ is a child of $u$, then $D_T(u) \geq 1+ \chi(uv) D_T(v)$.
\end{observation}

Note that descendant colourings avoid monochromatic copies of  $K_{1,3}$.
The above observation motivates the following definition. Given an unrooted tree $H$ and $s \in V(H)$, define $\mathcal{P}_s(H)$ as the set of all paths starting in $s$ and ending in a leaf of $H$, and let
\begin{equation}\label{eq:def_fH} f(H) = \min \Big\{ \max_{P \in \mathcal{P}_s(H)} \Big( \prod_{e \in P} \chi(e) \Big) \ \Big|\ \chi \colon E(H) \to \mathbb{Z}_{>0}\text{ is one-to-one and } s \in V(H)\Big\}.
\end{equation}
We have arrived at a key fact, which can be proved by repeatedly
applying Observation~\ref{obs:weight_consequence}.
\begin{lemma}\label{teo:cota_inferior}
Let $T$ be a rooted tree, $\chi$ be a descendant colouring of $T$ and
$H$ be an unrooted tree. If $\chi$ contains a rainbow copy of $H$, then $|V(T)| \ge f(H)$.
\end{lemma}

\begin{proof}
 Let $\varphi \colon V(H) \to V(T)$ be the immersion of a rainbow copy
 of $H$ in $T$, and let $s\in V(H)$ be the vertex that minimises
 $\operatorname{depth}(\varphi(s))$. We have from the definition of $f$ that
 \begin{equation}\label{eq:cota_inferior_copia}
 \max_{P \in \mathcal{P}_s(H)} \Big(\prod_{e\in P}\chi(e)\Big) \geq f(H).
 \end{equation}
 Let $P$ be an arbitrary path of $\mathcal{P}_s(H)$, write $P =
 v_0v_{1}\dots v_{k}$ with $v_0 = s$, and set $w_i = \varphi(v_i)$.
 By the choice of $s$, the vertex $w_i$ is a child of $w_{i-1}$ for every $1 \leq
 i \leq k$. Applying Observation~\ref{obs:weight_consequence} several times, we get that $D_T(\varphi(s)) = D_T(w_0) \geq (\prod_{i=1}^k \chi(w_{i-1 }w_i))\cdot D_T(w_k)$. Since $P$ was arbitrary and $D_T(w) \geq 1$ for all $w \in T$, we have that
 \[ |V(T)| \geq D_T(\varphi(s)) \geq \max_{P \in \mathcal{P}_s(H)} \Big( \prod_{i=1}^k \chi(w_{i-1}w_i ) \Big). \]
 Together with~\eqref{eq:cota_inferior_copia}, this completes the proof of the lemma.
\end{proof}

We will use Lemma~\ref{teo:cota_inferior} together with the following
result to obtain~\eqref{lower-mF}, which gives the lower bound for
$\mF(K_{1,3}, H)$ in the proof of Theorem~\ref{teo:arvore_binaria}.

\begin{lemma}\label{lem:f_geq_bin}
Let $B_{h+1} = (H, r)$ be the complete binary tree of height $h+1$. Then, \[f(H) \geq 2^{\binom{h}{2}}.\]
\end{lemma}

\begin{proof}
Denote by $T_1 = (H_1, r_1)$ and $T_2 = (H_2, r_2)$ the copies of
$B_{h}$ whose roots are the children of $r$. Note that for any vertex $s \in V(H)$, there exists $i \in \{1,2\}$ such that the path (in $H$) from $s$ to any vertex of $T_i$ passes through $r$. We will treat the case where $i = 1$, since the other case is symmetric. From Observation~\ref{obs:weight_consequence}, every $P \in \mathcal{P}' := \mathcal{P}_{r_1}(H_1)$ is contained in some path of $\mathcal{P}_s(H)$. Thus, by monotonicity, it suffices to show that, for every one-to-one function
$\chi \colon E(T_1) \to \mathbb{N}$, we have
\begin{equation}\label{eq:max_cota_inferior}
  \max_{P \in \mathcal{P}'} \left(\prod_{e \in P} \chi(e)\right) \geq 2^{\binom{h}{2}}.
\end{equation}
Since $T_1$ has $2^h$ leaves, $|\mathcal{P}'| = 2^h$. Using that the maximum of a sequence of numbers is at least their average, it is sufficient to show that
\begin{equation}\label{eq:soma_cota_inferior}
\sum_{P \in \mathcal{P}'} \sum_{e\, \in\, P} \log_2 \chi(e) \geq 2^h \binom{h}{2}.
\end{equation}
Since $e(T_1) = 2^{h+1} - 2$, we may assume that $\chi(E(T_1)) = \{1,
\dots, 2^{h+1}-2\} $, as changing one colour to another of higher
value clearly does not decrease the sum. Let $E_i \subset E(T_1)$ be
the set of edges connecting a vertex at depth $i-1$ to a vertex of
depth $i$ in $T_1$. Since each edge of $E_i$ is on $2^{h-i}$ paths of
$\mathcal{P}'$, we have
\begin{equation}\label{eq:soma_por_niveis}
\sum_{P \in \mathcal{P}'} \sum_{e\, \in\, P} \log_2 \chi(e) = \sum_{i=1}^h \sum_{e\, \in \, E_i} 2^{h-i} \log_2 \chi(e).
\end{equation}
By an exchange argument, this sum is minimised if and only if, for every $1 \le i < j \le h$ and any pair $(e_i, e_j) \in E_i \times E_j$, we have $\chi(e_i) < \chi(e_j)$. Since $|E_i| = 2^i$, a colouring that minimises the right-hand side of~\eqref{eq:soma_por_niveis} satisfies $\chi(E_i) = \{2^i-1, \dots, 2^{i+1} - 2 \}$. Thus,
\begin{equation}\label{eq:cota_soma_por_niveis}
\sum_{i=1}^h \sum_{e\, \in\, E_i} 2^{h-i} \log_2 \chi(e) \geq \sum_{i=1}^h 2^i \cdot 2 ^{h-i} \log_2 (2^{i-1}) = 2^h \sum_{i=1}^h (i-1) = 2^h \binom{h}{2}.
\end{equation}
Combining~\eqref{eq:soma_por_niveis} and~\eqref{eq:cota_soma_por_niveis} we get~\eqref{eq:soma_cota_inferior} and therefore~\eqref{eq:max_cota_inferior}, which we have already argued that implies the desired inequality. This finishes the proof.
\end{proof}

We are ready to prove Theorem~\ref{teo:arvore_binaria}, which
implies $\mF(K_{1,3}, H) = 1 - 2^{-(1+o(1))\binom{h}{2}}$ as previously discussed.

\begin{proof}[Proof of Theorem~\ref{teo:arvore_binaria}]
  To prove item \ref{lower-mF}, we will relate
  Lemmas~\ref{teo:cota_inferior} and \ref{lem:f_geq_bin}.  Let $F$ be
  a forest such that $F\mr(K_{1,3},H)$.  We root each component of $F$
  arbitrarily and colour them with a descendant colouring.  Let $T$ be
  a component that contains a rainbow copy of $H$.  By
  Lemmas~\ref{teo:cota_inferior} and~\ref{lem:f_geq_bin} we have
  $|V(T)| \geq f(H) \geq 2^{\binom{h-1}{2}}$. Since $F$ was arbitrary, we get the desired
  lower bound.

  To obtain item \ref{upper-mF}, let $G$ be a rooted tree of
  height $h$ where, for $0 \leq i < h$, the vertices at depth $i$ have
  $2^{i+3}$ children. Such a tree has $2^{\binom{h}{2}+3h}$ leaves, and
  therefore $2^{(1 +o(1))\binom{h}{2}}$ vertices. Consider a colouring of
  $E(G)$ without a monochromatic copy of $K_{1,3}$, and let $B_h$
  denote the complete rooted binary tree of height $h$. We claim that it
  is possible to find a rainbow copy of $B_h$ by embedding its
  vertices into $G$ by first embedding $B_0$ into the root of $G$ and
  then extending $B_i$ to $B_{i+1}$ for $0 \leq i < h$. Indeed, since
  each vertex at depth $i$ of $G$ is incident to edges of at least
  $2^{i+2}$ colours and $B_{i+1}$ has $2^{i+2}-2$ edges, there is a
  colour available when embedding the children of vertices at depth
  $i$.
\end{proof}

Adapting the proof of Lemma~\ref{lem:f_geq_bin}, one can show that $f(P_d)
\geq \sqrt{d!}$, where $P_d$ is the path with $d$ edges.
Analogously to the proof of Theorem~\ref{teo:arvore_binaria},
combining such a bound with a construction proves that
$\mF(K_{1,3}, P_d) = 1-d^{-(1/2+o(1))d}$.
We omit the details.

\bibliography{bibliography}

\end{document}